\documentclass{amsart}
\usepackage{latexsym}
\usepackage{amsfonts}
\usepackage{amssymb}

\begin{document}

\newcommand{\Q}{\mathbb{Q}}
\newcommand{\C}{\mathbb{C}}
\newcommand{\D}{\mathbb{D}}
\newcommand{\Z}{\mathbb{Z}}
\newcommand{\R}{\mathbb{R}}
\newcommand{\N}{\mathbb{N}}
\newcommand{\XTC}{\hat{\mathbb{C}}}
\newcommand{\hatnu}{\nu_{\XTC}}
\newcommand{\hatsigma}{\sigma_{\XTC}}
\newcommand{\deltaXTC}{\delta_{\XTC}}
\newcommand{\deltaopen}{\delta_{\mathcal{O}(\XTC)}}
\newcommand{\deltaopenen}{\delta_{\mathcal{O}(\XTC)}^{en}}
\newcommand{\deltaclosed}{\delta_{\mathcal{C}(\XTC)}}
\newcommand{\deltamero}{\delta_{\mathfrak{M}(\XTC)}}
\newcommand{\cat}{^\frown}

\newcommand{\dom}{\operatorname{dom}}
\newcommand{\ran}{\operatorname{ran}}
\newcommand{\diam}{\operatorname{diam}}
\newcommand{\normal}{n}
\newcommand{\Log}{\operatorname{Log}}
\newcommand{\Arg}{\operatorname{Arg}}
\renewcommand{\Re}{\operatorname{Re}}
\renewcommand{\Im}{\operatorname{Im}}
\newcommand{\Int}{\operatorname{Int}}
\newcommand{\Ext}{\operatorname{Ext}}
\newcommand{\ray}[1]{\overrightarrow{#1}}
\newcommand{\QED}{\box}
\newcommand{\const}{\mbox{\emph{Const.}} }

\newcommand{\converges}{\mathord{\downarrow}}
\newcommand{\diverges}{\mathord{\uparrow}}
\newcommand{\sub}[1]{_{\textup{\tiny{\fontfamily{cmr}\selectfont #1}}}}
\renewcommand{\leq}{\leqslant}
\renewcommand{\geq}{\geqslant}
\renewcommand{\le}{\leqslant}
\renewcommand{\ge}{\geqslant}
\renewcommand{\ngeq}{\ngeqslant}
\renewcommand{\nleq}{\nleqslant}

\renewcommand{\phi}{\varphi}

\newtheorem{theorem}{Theorem}[section]
\newtheorem{lemma}[theorem]{Lemma}
\newtheorem{assumption}[theorem]{Assumption}
\newtheorem{corollary}[theorem]{Corollary}
\newtheorem{exercise}[theorem]{Exercise}
\newtheorem{proposition}[theorem]{Proposition}
\newtheorem{givendata}[theorem]{Given Data}
\newtheorem{claim}[theorem]{Claim}
\newtheorem{property}[theorem]{Property}

\theoremstyle{definition}
\newtheorem{definition}[theorem]{Definition}
\newtheorem{question}[theorem]{Question}
\newtheorem{remark}[theorem]{Remark}
\newtheorem{convention}[theorem]{Convention}
\newtheorem{observation}[theorem]{Observation}
\newtheorem{problem}[theorem]{Problem}
\newtheorem{idea}[theorem]{Idea}
\newtheorem{notation}[theorem]{Notation}
\newtheorem{conjecture}[theorem]{Conjecture}
\newtheorem{example}[theorem]{Example}

\numberwithin{equation}{section}

\title[Asymptotic Density and the Coarse Computability
Bound]{Asymptotic Density and \\the Coarse Computability Bound}

\author[D. R. Hirschfeldt]{Denis R. Hirschfeldt}

\address{Department of Mathematics\\University of Chicago}
\email{drh@math.uchicago.edu}
\thanks{Hirschfeldt was partially supported by grant
DMS-1101458 from the National Science Foundation of the United
States.}

\author[C. G. Jockusch, Jr.]{Carl G. Jockusch, Jr.}

\address{Department of Mathematics\\University of Illinois at
Urbana-Cham\-paign}
\email{jockusch@math.uiuc.edu}

\author[T. H. McNicholl]{Timothy H. McNicholl}

\address{Department of Mathematics\\Iowa State University}
\email{mcnichol@iastate.edu}

\thanks{McNicholl was partially supported by a Simons Foundation Collaboration Grant for Math\-e\-ma\-ti\-cians}

\author[P. E. Schupp]{Paul E. Schupp}

\address{Department of Mathematics\\University of Illinois at
Urbana-Cham\-paign}

\email{schupp@illinois.edu}



\begin{abstract}
  For $r \in [0,1]$ we say that a set $A \subseteq \omega$ is
  \emph{coarsely computable at density} $r$ if there is a computable
  set $C$ such that $\{n : C(n) = A(n)\}$ has lower density at least
  $r$.  Let $\gamma(A) = \sup \{r : A \hbox{ is coarsely computable at
    density } r\}$.  We study the interactions of these concepts with
  Turing reducibility.  For example, we show that if $r \in (0,1]$
  there are sets $A_0, A_1$ such that $\gamma(A_0) = \gamma(A_1) = r$
  where $A_0$ is coarsely computable at density $r$ while $A_1$ is not
  coarsely computable at density $r$.  We show that a real $r \in
  [0,1]$ is equal to $\gamma(A)$ for some c.e.\ set $A$ if and only if
  $r$ is left-$\Sigma^0_3$.  A surprising result is that if $G$ is a
  $\Delta^0_2$ $1$-generic set, and $A \leq\sub{T} G$ with $\gamma(A)
  = 1$, then $A$ is coarsely computable at density $1$.
\end{abstract}
 
\maketitle

\section{Introduction}
   
   There are two natural models of ``imperfect computability''  
defined in terms of the standard notion of asymptotic density, which we
now review.  For $A \subseteq \omega$ and $n \in \omega \setminus
\{0\}$, define $\rho_n(A)$, the density of $A$ below $n$, by
$\rho_n(A) = \frac{|A \upharpoonright n|}{n}$, where $A
\upharpoonright n = A \cap \{0, 1, \dots, n-1\}$.  Then 
\[ \underline{\rho}(A) = \liminf_n \rho_n(A)\qquad \mbox{and}\qquad \overline{\rho}(A) = \limsup_n \rho_n(A) \]
 are respectively the \emph{lower density} of $A$
 and the \emph{upper density} of $A$.
  The
\emph{(asymptotic) density} of $A$ is  $\rho(A) = \lim_n \rho_n(A)$
 provided the limit exists.

     The idea of generic computability was introduced and studied in
  connection with group theory in \cite{KMSS} and then studied in
  connection with arbitrary subsets of $\omega$ in \cite{JS}. In
  generic computability we have a partial algorithm that is always
  correct when it gives an answer but may fail to answer on a set of
  density $0$.  The paper \cite{DJS} began studying computability at
  densities less than $1$ and introduced the following definitions.

\begin{definition}[{\cite[Definition 5.9]{DJS}}]
\label{def: Computable.at. r} 
  Let $A$ be a set of natural numbers  and let $r$ be a real number in
  the unit interval $[0,1]$.  The set  $A$ is \emph{partially computable at density
    $r$} if there is a partial computable function $\phi$ such that
  $\phi(n) = A(n)$ for all $n$ in the domain of $\phi$ and the
  domain of $\phi$ has lower density at least $r$.
\end{definition}

Thus  $A$ is \emph{generically computable} if and only if $A$ is
partially computable at density $1$.   

\begin{definition}[{\cite[Definition 6.9]{DJS}}]
\label{def:ALPHA}
If $A \subseteq \omega$, the \emph{partial  computability bound} of $A$ is
\[ \alpha(A) = \sup\{r : \mbox{$A$ is  partially computable at density $r$}\}. \] 
\end{definition}

  In the paper \cite{DJS} the term  ``partially computable at density $r$'' was
simply called ``computable at density $r$'' and the   ``partial computability bound''
 was called the ``asymptotic computability bound''.    That paper considered only 
partial computability at densities less than $1$, but since  we are here comparing the
partial computability concepts with their  coarse analogs,  the present terminology
is more exact.

If $A$ is generically computable, then $\alpha(A) = 1$.  The converse
fails by \cite[Observation 5.10]{DJS}.  There are sets that are partially 
computable at every density less than $1$ but are not generically
computable.

\begin{definition}\label{def:COARSESIM}
  If $A, B \subseteq \N$, then $A$ and $B$ are \emph{coarsely
    similar}, written $A \backsim\sub{c} B$, if the density of the
  symmetric difference of $A$ and $B$ is $0$, that is, $ \rho(A
  \triangle B) = 0$.  Given $A$, any set $B$ such that $B \backsim\sub{c} A$
   is called a \emph{coarse description} of $A$.
\end{definition}

It is easy to check that coarse similarity is indeed an equivalence
relation.  Coarse similarity was called \emph{generic similarity} in
\cite{JS}, but the  current terminology seems better.

     Coarse computability  considers algorithms that
always give an answer, but may give an incorrect answer on a set of
density $0$.  We have the following definition.

\begin{definition}[{\cite[Definition 2.13]{JS}}]
\label{def:COARSE}
  The set $A$ is \emph{coarsely computable} if there is a computable
  set $C$ such that the density of $\{n : A(n) = C(n)\}$ is $1$.
  That is, $A$ is coarsely computable if it has a computable coarse description $C$.
\end{definition}

  The following definitions are similar to those for partial computability.

  \begin{definition}\label{def:COARSE.AT.r}
    If $A \subseteq \omega$ and $r \in [0,1]$, an \emph{$r$-description} of $A$
    is  any set $B$ such that the lower density of $\{n : A(n) = B(n)\}$ is
    at least $r$.  A set $A$ is  \emph{coarsely computable at density $r$} if there is a 
    computable $r$-description $B$ of $A$.
\end{definition}

   Note that  $A$ is coarsely computable if and only $A$ is coarsely computable
at density $1$.

\begin{definition}\label{def:GAMMA}
  If $A \subseteq \omega$, the \emph{coarse computability bound} of
  $A$ is
\[ \gamma(A) = \sup\{r : \mbox{$A$ is  coarsely computable at density  $r$}\}. \] 
\end{definition}

If $A$ is coarsely computable, then $\gamma(A) = 1$, but the next
lemma implies that the converse fails.

It is shown in \cite[Proposition 2.15 and Theorem 2.26]{JS} that
neither of generic computability and coarse computability implies the
other, even for c.e.\ sets.  Nonetheless, the following lemma gives an
inequality between $\alpha$ and $\gamma$.

\begin{lemma} \label{ineq} For any $A \subseteq \omega$, $\alpha(A)
  \le \gamma(A)$.  In particular, if $A$ is generically computable
  then $\gamma(A) = 1$.
\end{lemma}

\begin{proof} Fix $\epsilon > 0$.  If $\alpha(A) = r$ then there is a
  partial algorithm $\phi$ for $A$ such that the lower density of
  the c.e.\ set $ D = \dom\phi$ is greater than or equal to
  $r - \epsilon$.  Theorem 3.9 of \cite{DJS} shows that if $D$ is a
  c.e.\ set there is a computable set $C \subseteq D$ such that
  $\underline{\rho}(C) > \underline{\rho}(D) - \epsilon$.  Let $C_1 =
  \{n \in C : \phi(n) = 1\}$.  Then $C_1$ is a computable set and
  $\{n : A(n) = C_1(n)\} \supseteq C$.  It follows that
  $\underline{\rho}(\{n : A(n) = C_1(n)\}) \geq \underline{\rho}(C) >
  \underline{\rho}(D) - \epsilon \geq r - 2\epsilon$, and hence $A$ is
  coarsely computable at density $r - 2\epsilon$.  Since $\epsilon > 0$
  was arbitrary, it follows that $\gamma(A) \geq r = \alpha(A)$.
\end{proof}

One consequence of this lemma is that any set that is generically
computable but not coarsely computable is an example of a set $A$ such
that $\gamma(A)=1$ but $A$ is not coarsely computable.

\begin{definition}\label{def:METRIC}
If $A, B \subseteq \N$,  let $D(A,B) = \overline{\rho}(A \triangle B )$.
\end{definition}

It is shown in \cite[remarks after Proposition 3.2]{DJS} that $D$ is
a pseudometric on subsets of $\omega$ and, since $D(A,B) = 0$ exactly
when $A$ and $B$ are coarsely similar, $D$ is actually a metric on the
space of coarse similarity classes.  Note that $\gamma$ is an
invariant of coarse similarity classes.

 Although easy,  the following is useful enough to be stated as a lemma.
 
 \begin{lemma} \label{comp} If $A \subseteq \omega$ then $
   \underline{\rho} (A) = 1 - \overline{\rho}(\overline A) $.
\end{lemma}

\begin{proof} Note  that $\rho_n(A) = 1 - \rho_n(\overline{A})$ for all $n \ge 1$.
  The lemma follows by taking the lim inf of both sides of this
  equation.
\end{proof}

   Since we have a pseudometric space, we can consider the distance from
a single point to a subset of the space in the usual way.

\begin{definition}\label{SETDISTANCE}
If  $A \subseteq \omega$ and $\mathcal{S} \subseteq \mathcal{P}(\N)$, let 
\[
\delta(A, \mathcal{S}) = \inf\{D(A, S) : S \in \mathcal{S}\}.
\]
\end{definition}

 The above lemma shows that 
\[
\gamma(A) = 1 - \delta(A, \mathcal{C}),
\]
where $\mathcal{C}$ is the class of computable sets. Thus $\gamma(A) =
1 $ if and only if $A$ is a limit of computable sets in the
pseudometric.  A set $A$ is coarsely computable at density $r$ if and only
if $\delta(A,\mathcal{C}) \leq 1 - r$.

 The symmetric difference  $ A \triangle B = \{ n: A(n) \ne B(n) \} $ is
 the subset of $\omega$ where $A$ and $B$ disagree. There does not seem to be 
 a standard notation
for the complement of $ A \triangle B $, which is $ \{ n: A(n) = B(n)\} $, 
the ``symmetric agreement'' of $A$ and $B$.  We find it useful
to use $ A \triangledown B $ to denote $\{n : A(n) = B(n)\}$.

We assume that the reader is familiar with basic computability theory.
See, for example, \cite{S}.  If $S$ is a set of finite binary strings and $A
\subseteq \omega$ we say that $A$ \emph{meets} $S$ if $A$ extends some
string in $S$ and that $A$ \emph{avoids} $S$ if $A$ extends a string
that has no extension in $S$.

\section{Turing degrees, coarse computability, and $\gamma$}

It is easily seen that every Turing degree contains a set that is
both coarsely and generically computable  and hence a set
$A$ with $\alpha(A) = \gamma(A) = 1$.   In the other direction
   it is shown in Theorem 2.20 of \cite{JS} that every nonzero Turing
   degree contains a set that is neither generically computable
   nor coarsely computable.  The same construction now yields a
   quantitative version of that result.

 \begin{theorem} \label{I(A)} Every nonzero Turing degree contains a
     set whose partial computability bound is $0$ but whose coarse computability bound is $1/2$.
\end{theorem}

\begin{proof} Let $I_n = [n!,(n+1)!)$. Suppose that $A$ is not
  computable, and let $\mathcal{I}(A) = \bigcup_{n \in A} I_n$.  It is
  clear that $\mathcal{I}(A)$ is Turing equivalent to $A$.  We prove
  first that $\gamma(\mathcal{I}(A)) \le \frac{1}{2}$.  If there is a
  computable $C$ with $\underline{\rho}(\mathcal{I} (A) \triangledown
  C) > \frac{1}{2}$ we can compute $A$ by ``majority vote''.  That is, for all
  sufficiently large $n$, we have that $n$ is in $A$ if and only if
  more than half
  of the elements of $I_n$ are in $C$.  (For any $n$ for which this
  equivalence fails, we have $\rho_{(n+1)!}(\mathcal{I}(A)
  \triangledown C) \leq (1 + (n+1)^{-1})/2$.)  It follows that $A$ is
  computable, a contradiction.  If $C$ is the set of even numbers,
  then it is easily seen that $\rho(C \triangledown \mathcal{I}(A)) =
  \frac{1}{2}$, so $\gamma(\mathcal{I}(A)) \geq \frac{1}{2}$.  It follows that
  $\gamma(\mathcal{I}(A)) = \frac{1}{2}$.  To see that $\alpha(\mathcal{I}(A))
  = 0$, note that any set of positive lower density intersects $I_n$
  for all but finitely many $n$, and apply this observation to the
  domain of any partial computable function that agrees with
  $\mathcal{I}(A)$ on its domain.
\end{proof}

We next observe that a large class of degrees contain sets $A$ with
$\gamma(A) = 0$.

\begin{theorem} \label{hi} 
Every hyperimmune degree contains a set whose coarse computability bound is $0$.  
\end{theorem}

\begin{proof}

  A set $S \subseteq 2^{ < \omega}$ of finite binary strings is
  \emph{dense} if every string has some extension in $S$.  Stuart
  Kurtz \cite{K} defined a set $A$ to be \emph{weakly $1$-generic} if
  $A$ meets every dense c.e.\ set $S$ of finite binary strings
  and  proved that the weakly $1$-generic degrees coincide with
  the hyperimmune degrees.  Hence, it suffices to show that every
  weakly $1$-generic set $A$ satisfies $\gamma(A) = 0$.  Assume that
  $A$ is weakly $1$-generic.

  If $f$ is a computable function then, for each $n,j > 0$, define
  \[
S_{n,j} = \left\{ \sigma \in 2^{< \omega} : |\sigma| \ge j \enspace \&
\enspace \rho_{|\sigma|}(\{k < |\sigma| : \sigma(k) = f(k)\}) <
\frac{1}{n} \right\}.
\]
Each set $S_{n,j}$ is computable and dense so $A$ meets each
$S_{n,j}$.  Thus $\{ k: f(k) = A(k)\}$ has lower density $0$.
\end{proof}

In view of the preceding result, it is natural to ask whether \emph{every}
nonzero degree contains a set $A$ such that $\gamma(A) = 0$.
This question is answered in the negative in \cite{ACDJL} where it is
shown that that every computably traceable set is coarsely computable
at density $\frac{1}{2}$, and also that every set computable from a $1$-random
set of hyperimmune-free degree is coarsely computable at density
$\frac{1}{2}$.  Each of these results implies that there is a nonzero degree
$\mathbf{a \leq 0''}$ such that every $\mathbf{a}$-computable set is
coarsely computable at density $\frac{1}{2}$.  Here it is not possible to
replace $\frac{1}{2}$ by any larger number, by Theorem \ref{I(A)}.  In
\cite{ACDJL}, the following definition is made for Turing degrees
$\mathbf{a}$:
$$\Gamma(\mathbf{a}) = \inf \{\gamma(A) :  A \mbox{ is 
$\mathbf{a}$-computable} \}.$$
By the above, $\Gamma$ takes on the values $0$ and $\frac{1}{2}$, and of
course $\Gamma(\mathbf{0}) = 1$.  By
Theorem \ref{I(A)}, $\Gamma$ does not take on any values in the open
interval $(\frac{1}{2}, 1)$.   An open question posed  in \cite{ACDJL}
is whether $\Gamma$ takes on any values other than $0, \frac{1}{2}$, and $1$.

\section{Coarse computability at density $\gamma(A)$}

If $A$ is any set, it follows from the definition of $\gamma(A)$ that
$A$ is coarsely computable at every density less than $\gamma(A)$
and at no density greater than $\gamma(A)$.  
What happens at $\gamma(A)$?  Let us say that $A$ is \emph{extremal for coarse computability} if it is coarsely computable at density $\gamma(A)$.
In this section, we show that extremal and non-extremal sets exist.  Moreover, we also show that every real in $(0,1]$ is the coarse computability bound of an extremal set and of a non-extremal set.  We also explore the distribution of these cases in the Turing degrees.   Roughly speaking, we show that hyperimmune degrees yield extremal sets and high degrees yield non-extremal sets.

\begin{theorem} \label{prescribed} 
Every real in $[0,1]$ is the coarse computability bound of a set that is extremal for coarse computability.  
\end{theorem}

\begin{proof}  Suppose $0 \leq r \leq 1$.  By Corollary 2.9 of \cite{JS} there is a set $A_1$ such
  that $\rho(A_1) = r$.  Let $Z$ be a set with $\gamma(Z) = 0$, which
  exists by Theorem \ref{hi}, and let $A = A_1 \cup Z$.  Note
  first that that $A$ is coarsely computable at density $r$ via the
  computable set $\omega$ since
$$\underline{\rho}(A \triangledown \omega) = \underline{\rho}(A) 
\geq \underline{\rho}(A_1) = r.$$ 
It follows that $\gamma(A) \geq r$,
so it remains only to show that $\gamma(A) \leq r$.

Suppose for a contradiction that $\gamma(A) > r$, so $A$ is coarsely
computable at some density $r' > r$. Let $C$ be a computable set such
that $\underline{\rho}(A \triangledown C) \ge r'$.  Let:
\begin{eqnarray*}
S_1 & = & A_1 \cap C\\ 
S_2 & = &  (Z \setminus A_1) \cap
C\\
S_3 & = & \overline{A} \cap \overline{C}.
\end{eqnarray*}
  Note that $A
\triangledown C$ is the disjoint union of $S_1$, $S_2$, and $S_3$ so
\[ \rho_n(A \triangledown C) = \rho_n(S_1) + \rho_n(S_2) + \rho_n(S_3) \]
for all $n$.

Let $\epsilon = r' - r$.  For all sufficiently large $n$ we have
$\rho_n(A \triangledown C) > r + \frac{\epsilon}{2}$.  Since $S_1
\subseteq A_1$ and $\rho_n(A_1) < r + \frac{\epsilon}{3}$ for all
sufficiently large $n$, we have
$\rho_n(S_2) + \rho_n(S_3) > \frac{\epsilon}{6}$ for all sufficiently
large $n$.  Hence $\underline{\rho}(S_2 \cup S_3) > 0$.  But $S_2
\cup S_3 \subseteq C \triangledown Z$ so $\underline{\rho}(C
\triangledown Z) > 0$, contradicting $\gamma(Z) = 0$.  This
contradiction shows that $\gamma(A) \leq r$, and the proof is complete.
\end{proof}

\begin{corollary}[to proof] 
Suppose $\mathbf{a}$ is a hyperimmune degree.  Then, every $\Delta^0_2$ real in $[0,1]$ is the coarse computability bound of a set in $\mathbf{a}$ that is extremal for coarse computability.
\end{corollary}

\begin{proof} Just note that the proof of the theorem can be carried
  out effectively in $\mathbf{a}$.  In more detail, by Theorem 2.21 of
  \cite{JS} there is a computable set $A_1$ of density $r$.  Further,
  by Theorem \ref{hi} there is an $\mathbf{a}$-computable set $Z$
  such that $\gamma(Z) = 0$.  Then $A = A_1 \cup Z$ satisfies the
  theorem and is $\mathbf{a}$-computable.  We can ensure that $A \in \mathbf{a}$ by coding a set in $\mathbf{a}$ into $A$ on a set of density $0$.
\end{proof}

We now consider sets that are not extremal for coarse computability.  We first consider the degrees
of the sets $A$ such that $\gamma(A) = 1$ but $A$ is not coarsely
computable.

Define 
$$R_n = \{k : 2^n \mid k \enspace \& \enspace 2^{n+1} \nmid k\}.$$
The sets $R_n$ were heavily used in \cite{JS} and \cite{DJS}.  Note
that they are uniformly computable and pairwise disjoint, and
$\rho(R_n) = 2^{-(n+1)}$.  As in \cite{JS} and \cite{DJS}, define
$$\mathcal{R}(A) = \bigcup_{n \in A} R_n.$$
Note that, for all $A$, we have that $A \equiv\sub{T} \mathcal{R}(A)$ and
$\alpha(\mathcal{R}(A)) = \gamma(\mathcal{R}(A)) = 1$.  To see the
latter (which was pointed out by Asher Kach), note that if $C_k =
\bigcup \{R_n : n \in A \enspace \& \enspace n < k\}$, then $C_k$ is computable and agrees
with $\mathcal{R}(A)$ on $\bigcup_{n < k} R_n$, and the latter has density
$1 - 2^{-k}$.

\begin{theorem} \label{degreeobs} 
     \begin{itemize}
            \item[(i)]
If $\boldsymbol{a}$ is a degree such that $\mathbf{a \nleq 0'}$,
  then $\boldsymbol{a}$ contains a set that is not coarsely computable but whose coarse computability bound is $1$.  
  
\item[(ii)]
 If $\boldsymbol{a}$ is a nonzero c.e.\ degree, then
  $\boldsymbol{a}$ contains a c.e.\ set that is not coarsely computable but whose coarse computability bound is $1$.
        \end{itemize}
\end{theorem}

\begin{proof}  It is shown in Theorem 2.19 of \cite{JS} that
  $\mathcal{R}(B)$ is coarsely computable if and only if $B$ is
  $\Delta^0_2$.  If $\mathbf{a \nleq 0'}$ and $B$ has degree
  $\mathbf{a}$, then $\mathcal{R}(B)$ is a set of degree
  $\mathbf{a}$ that is not coarsely computable even though its coarse computability bound is $1$.  Part (i) follows.

 Theorem 4.5 of \cite{DJS} shows that every nonzero c.e.\
  degree contains a c.e.\ set $A$ that is generically computable but not
  coarsely computable.  Then $\alpha(A) = 1$, so by Lemma \ref{ineq},
  $\gamma(A) = 1$.   This proves part (ii).
\end{proof}

This result raises the natural question: Does \emph{every} nonzero Turing
degree  contain a set $A$ such
that $\gamma(A) = 1$ but $A$ is not coarsely computable?  We will
obtain a negative answer  in Theorem
\ref{1G} in the next section.   In fact, we will show that if $G$ is
$1$-generic and $\Delta^0_2$, and  $A \leq\sub{T} G$ has $\gamma(A) = 1$, then
$A$ is coarsely computable. 

We now consider the coarse computability bounds of non-extremal sets.

\begin{theorem} \label{notcc} 
Every real in $(0,1]$ is the coarse computability bound of a set that is not extremal for coarse computability. 
\end{theorem}

\begin{proof}
Suppose $0 < r \leq 1$.  We construct a set $A$ so that $\gamma(A) = r$ but $A$ is not coarsely computable at density $r$.  
As an auxiliary for defining $A$, we first use the technique of Corollary 2.9 of \cite{JS} to define a set $S$ of density $r$.  
To this end, we turn $r$ into a set $B$ in the natural way.  
That is, since $r > 0$, it has a non-terminating binary expansion $r = 0.b_0b_1 \dots$.  
We then set $B = \{i : b_i = 1\}$.  By restricted countable additivity (Lemma 2.6 of \cite{JS}), $\mathcal{R}(B)$ has density $r$.  Set $S = \mathcal{R}(B)$.

We now divide $S$ into ``slices'' $S_0, S_1, \ldots$ as follows.  Let $c_0 < c_1 < \cdots$ be the increasing enumeration of $B$.  Set $S_e = R_{c_e}$.  Note that the $S_e$'s are pairwise disjoint and that $S = \bigcup_e S_e$.  Note also that each $S_e$ is computable (though not necessarily computable uniformly in $e$).  

We now define $A$.  We first choose a set $Z$ so that $\gamma(Z) = 0$.  Such a set exists by Theorem \ref{hi}.  Let $C_0, C_1, \ldots$ be an enumeration of the computable sets.  We then set 
\[
A = (\overline{S} \cap Z) \cup \bigcup_e (S_e \cap \overline{C_e}).
\]
 
We now claim that $A$ is coarsely computable at density $q$ whenever $0 \leq q < r$.  For, suppose $0 \leq q < r$.  Since the density of $S$ is $r$, there is a number $n$ so 
that $\rho(\bigcup_{e < n} S_e) \geq q$.  Let $C = \bigcup_{e < n} (S_e \cap
\overline{C_e})$.  Then, $C$ is a computable set.  Also $A$ and $C$ agree on each $S_e$ for $e
< n$, so $\underline{\rho}(A \triangledown C) \geq
\underline{\rho}(\bigcup_{e < n} S_e) \geq q$.  Hence, $C$ witnesses that $A$
is coarsely computable at density $q$. 

To complete the proof, it suffices to show that $A$ is not coarsely computable at  density $r$.  To this end, it suffices to show that the lower density of $A \triangledown C_e$ is smaller than $r$ for each $e$.  Fix $e \in \N$.  By construction, $(C_e
  \triangledown A) \cap S$ is disjoint from $S_e$ and so has
  \emph{upper} density less than $r$.  At the same time, note that $(A \triangledown C_e) \cap \overline{S} \subseteq C_e \triangledown Z$.  Let $r_0 =
  \overline{\rho}((C_e \triangledown A) \cap S)$, and let $\epsilon = r
  - r_0$.   Then for infinitely many $n$ we have
$$\rho_n(A \triangledown C_e) = \rho_n((A \triangledown C_e) \cap S) +
\rho_n((A \triangledown C_e) \cap \overline{S}) < \left(r_0 + \frac{\epsilon}{2}\right) +
\frac{\epsilon}{3} < r.$$ It follows that $\underline{\rho}(A \triangledown C_e) < r$.
Hence $A$ is not coarsely computable at density $r$, which completes
the proof.
\end{proof}

We note that the proof of Theorem \ref{notcc} shows that if $A$ is any set so that $A \cap S_e = \overline{C_e} \cap S_e$ for all $e$, then $A$ is computable at density $q$ whenever $0 \leq q < r$.  That is, the construction of $A \cap S$ ensures that $\gamma(A) \geq r$.

\begin{corollary}[to proof] 
Suppose $\mathbf{a}$ is a high degree.  Then, every computable real in $(0,1]$ is the coarse computability bound of a set in $\mathbf{a}$ that is not extremal for coarse computability.
\end{corollary}

\begin{proof} 
  We just observe that the preceding proof can be carried out in an
  $\mathbf{a}$-com\-put\-able fashion.  By Theorem 1 of \cite{J2}, there
  is a listing $C_0, C_1, \dots$ of the computable sets that is
  uniformly $\mathbf{a}$-computable. Also, since $r$ is computable,
  the sequence $S_0, S_1, \dots$ in the proof of the theorem is also
  uniformly $\mathbf{a}$-computable.  It does not affect the proof to
  modify each $S_e$ so that it contains no numbers less than $e$, and
  then $S = \bigcup_e S_e$ is $\mathbf{a}$-computable.  Finally, every
  high degree is hyperimmune by a result of D. A. Martin \cite{M},
  and so every high degree computes a set $Z$ with $\gamma(Z) = 0$ by
  Theorem \ref{hi}.  Hence the set $A$ defined in the proof of the
  theorem can be chosen to be $\mathbf{a}$-computable.  By coding a set in $\mathbf{a}$ into $A$ on a set of density $0$ we can ensure that $A \in \mathbf{a}$.
\end{proof}

By using suitable computable approximations, the previous corollary
can be extended from computable reals to 
$\Delta^0_2$ reals.  We omit the details.

It was shown in Theorem 4.5 of \cite{DJS} that every nonzero
c.e.\ degree contains a c.e.\ set that is generically computable but
not coarsely computable.  It follows at once from Lemma \ref{ineq}
that every nonzero c.e.\ degree contains a c.e.\ set $A$ such that
$\gamma(A) = 1$ but $A$ is not coarsely computable.  We now use the
method of Theorem \ref{notcc} to extend this result to the case where
$\gamma(A)$ is a given computable real.

\begin{theorem} 
Suppose $\mathbf{a}$ is a nonzero c.e.\ degree.  Then, every computable real in $(0,1]$ is the coarse computability bound and the partial computability bound of a c.e.\ set in $\mathbf{a}$ that is not extremal for coarse computability.
\end{theorem}

\begin{proof}
  Define the sets $S, S_0, S_1, \ldots$ as in the proof of Theorem \ref{notcc} so that $S = \bigcup_e S_e$ and so that $\rho(S) = r$.  Let $B$ be
  a c.e.\ set of degree $\mathbf{a}$, and let $\{B_s\}$ be a computable
  enumeration of $B$.  We construct the desired set $A \leq\sub{T} B$ using
  ordinary permitting; i.e.\ if $x \in A_{s+1} \setminus A_s$, then there exists $y \leq x$ such that $y \in B_{s+1} \setminus B_s$.  To ensure that $B \leq\sub{T} A$, we code $B$ into $A$ on a set of density zero.

  Let the requirement $N_e$ assert that if $\Phi_e$ is total, then the lower density of the
  set on which it agrees with $A$ is smaller than $r$.  Thus, if $N_e$ is met for every $e$, then $A$ is not coarsely computable at density $r$.  We
  meet $N_e$ by appropriately defining $A$ on $S_e$ and on
  $\overline{S}$.  If $\Phi_e$ is total, we meet $N_e$ by making $A$ completely
disagree with $\Phi_e$ on infinitely many large finite sets $I \subseteq
S_e \cup \overline{S}$.    To this end, we effectively choose finite sets $I_{e,i}$ such that the following hold for all $e$, $i$, $e'$, and $i'$:
\begin{enumerate}
    \item[(i)]  $I_{e,i} \subseteq (S_e \cup \overline{S})$.
     \item[(ii)] $\min I_{e, i+1} > \max I_{e,i}$.
    \item[(iii)] $\rho_m(I_{e,i}) \geq \frac{i}{i+1} \rho_m(S_e \cup
      \overline{S})$ where $m = \max I_{e,i} + 1$.
      \item[(iv)] If  $(e,i) \neq (e', i')$, then $I_{e,i} \cap I_{e', i'}
= \emptyset$.
\end{enumerate}

The sets $I_{e,i}$ may be obtained by intersecting appropriately large
intervals with $S_e \cup \overline{S}$ while preserving pairwise
disjointness, and we will call the sets $I_{e,i}$ ``intervals''.
During the construction we will designate an interval $I_{e,i}$ as
``successful'' if we have ensured that $\Phi_e$ and $A$ totally
disagree on $I_{e,i}$.  The construction is as follows:

\emph{Stage} $0$.   Let $A_0 = \emptyset$.

\emph{Stage} $s + 1$.  For each $e, i \leq s$, declare $I_{e,i}$ to be successful if it has not yet been declared successful and if the following conditions are met.
\begin{enumerate}
	\item $\Phi_{e,s}$ is defined on all elements
of $I_{e,i}$.

	\item $\min(I_{e,i})$ exceeds all elements of $A_s \cap S_e$.\label{itm:inc}
	
	\item At least one number in $B_{s+1} \setminus B_s$ is less than or equal to $\min(I_{e,i})$.
\end{enumerate}
If $I_{e,i}$ is declared to be successful at stage $s + 1$, then enumerate into $A$ all $x \in I_{e,i}$ with $\Phi_e(x) = 0$.

The set $A$ is clearly c.e., and $A \leq\sub{T} B$ by ordinary permitting.
If the interval $I_{e,i}$ is ever declared to be successful, then $A$
and $\Phi_e$ totally disagree on $I_{e,i}$, by the action taken
when it is declared successful and the disjointness condition (iv),
which ensures that no elements of $I_{e,i}$ are enumerated into $A$
except by this action.

Note that (\ref{itm:inc}) ensures that $A \cap S_e$ is computable for each $e$.  It follows that $\gamma(A) \geq \alpha(A) \geq r$ as in the
proof of Theorem \ref{notcc}.
 
It remains to show that every requirement $N_e$ is met.  Suppose that $\Phi_e$ is
total.  We claim first that the interval $I_{e,i}$ is declared
successful for infinitely many $i$.  Suppose not.  Then $A \cap S_e$
is finite.  It follows that $B$ is computable, since, for all
sufficiently large $i$, if $s \geq i$ and $\Phi_{e, s}$ is defined on
all elements of $I_{e,i}$, then no number less than $\min(I_{e,i})$
enters $B$ after stage $s$.  Since we assumed that $B$ is
noncomputable, the claim follows.

Suppose $I_{e,i}$ is successful.  Set $I = I_{e,i}$.  Then $A \triangle \Phi_e \supseteq I$, so
$$\rho_m (A \triangle \Phi_e) \geq \rho_{m}(I)  \geq 
\frac{i}{i+1} \rho_{m}(S_e \cup \overline{S}),$$ where $m = \max
I_{e,i} +1$.  There are infinitely many such $i$'s, and as $i$ tends
to infinity, the right hand side of the above inequality tends to
$\rho(S_e) + \rho(\overline{S})$.  It follows that $\overline{\rho}(A
\triangle \Phi_e) \geq \rho(S_e) + (1 - r)$, and so by Lemma \ref{comp},
$\underline{\rho}(A \triangledown \Phi_e) \leq r - \rho(S_e) < r$, as needed to
complete the proof.
\end{proof}

\section{Coarse Computability and Lowness}

We now consider the coarse computability properties of c.e.\ sets that
have a density.

\begin{proposition} \label{c.e.density} Every low c.e.\ set having a
  density is coarsely computable.  Every c.e.\ set having a density
  has coarse computability bound $1$.
\end{proposition}

\begin{proof} The first statement is Corollary 3.16 of \cite{DJS}.
  Let $A$ be a c.e.\ set that has a density and let $\epsilon > 0$.
  Theorem 3.9 of \cite{DJS} shows that $A$ has a computable subset $C$
  such that $\underline{\rho}(C) > \rho(A) - \epsilon$.  Then $C
  \triangle A = A \setminus C$.  Hence, by Lemma \ref{comp},
  $\underline{\rho}(A \triangledown C) = 1 - \overline{\rho}(A
  \setminus C)$.  But by Lemma 3.3 (iii) of \cite{DJS},
$$\overline{\rho}(A \setminus C) \leq \rho(A) - \underline{\rho}(C) < 
\epsilon.$$   
Hence $\underline{\rho}(A \triangledown C) > 1 - \epsilon$.  Since
$\epsilon > 0$ was arbitrary, we conclude that $\gamma(A) = 1$.
\end{proof}
 
The next result shows that the lowness
assumption is strongly required in the first part of Proposition \ref{c.e.density}.

\begin{theorem}\label{thm:nonlow}
Every nonlow c.e.\ Turing degree $\mathbf{a}$ contains
  a c.e.\ set of density $1/2$ that is not coarsely computable.
\end{theorem}

\begin{proof} The proof of the theorem is similar to the proof in
  Theorem 4.3 of \cite{DJS} that every nonlow c.e.\ degree contains a
  c.e.\ set $A$ such that $\rho(A) = 1$ but $A$ has no computable
  subset of density $1$.  Hence we give only a sketch.  Let $C$ be a
  c.e.\ set of degree $\mathbf{a}$.  We ensure that $A \leq\sub{T} C$ by a
  slight variation of ordinary permitting: If $x$ enters $A$ at stage
  $s$, then either some number $y \leq x$ enters $C$ at $s$ or $x =
  s$.  This implies that $A \leq\sub{T} C$, and by coding $C$ into $A$ on a
  set of density $0$ we can ensure that $A \equiv\sub{T} C$ without
  disturbing the other desired properties of $A$.

  To ensure that $\rho(A) = \frac{1}{2}$, we arrange that $\rho(A \cap R_n) =
  \frac{\rho(R_n)}{2}$ for all $n$.  Then by restricted countable additivity
  (Lemma 2.6 of \cite{JS}), 
$$\rho(A) = \sum_n \rho(A \cap R_n) = \sum_n \frac{\rho(R_n)}{2} = \frac{\sum_n \rho(R_n)}{2} = \frac{1}{2}.$$
Let $R_n$ be listed in increasing order as $r_{n,0}, r_{n,1}, \dots$.
We require that, for all $n$ and all sufficiently large $k$, exactly
one of $r_{n, 2k}$ and $r_{n,2k+1}$ is in $A$.  This clearly implies
that $\rho(A \cap R_n) = \frac{\rho(R_n)}{2}$.

Let $N_e$ be the requirement that $\overline{\rho}(A \triangle \Phi_e)
> 0$ if $\Phi_e$ is total.  So, if $N_e$ is met, then $A$ is not
coarsely computable via $\Phi_e$.  We will define a ternary computable
function $g(e,i,s)$ to help us meet this requirement by
``threatening'' to witness that $C$ is low.  Let $N_{e,i}$ be the
requirement that either $N_e$ is met or $C'(i) = \lim_s g(e,i,s)$.
Since $C$ is not low, to meet $N_e$ it suffices to meet all of its
subrequirements $N_{e,i}$.  Let $R_{e,i}$ denote $R_{\langle e, i
  \rangle}$.  We use $R_{e,i}$ to meet $N_{e,i}$.
  
Fix $e,i$.  Our module for satisfying $N_{e,i}$ proceeds as follows.
Let $s_0$ be the least number so that $\Phi_{i,s_0}(C_{s_0}; i)
\converges$; if there is no such number, then let $s_0 = \infty$. For
each $s < s_0$, let $g(e,i,s) = 0$ and put $s$ into $A$ if $s$ is of
the form $r_{\langle e,i \rangle, 2k}$ for some $k$.  If $s_0$ is
infinite, that is if the search for $s_0$ fails then no other work is done
on $N_{e,i}$. (Note that in this case $\lim_s g(e,i,s) = 0 = C'(i)$,
so $N_{e,i}$ is met.) Suppose $s_0$ is finite (that is, the search for
$s_0$ succeeds).  We choose an interval $I_0 \subseteq R_{e,i}$ as
follows.  Let $I_0$ be of the form $\{r_{\langle
  e,i\rangle, 2j}, \ldots, r_{\langle e,i \rangle, 2k+1} \}$ so that
$\min(I_0) > s_0$ and so that $\rho_m(I_0) \geq \rho_m(R_{e,i})/2$
where $m = r_{\langle e,i \rangle, 2k + 1} + 1$.  Let $u_0$ be the use
of the computation $\Phi_{i,s_0}^{C_{s_0}}(i)$.  Note that $u_0 < s_0$
by a standard convention and that no element of $I_0$ has been
enumerated in $A$.  We then restrain all elements of $I_0$ from
entering $A$ but continue putting alternate elements of $R_{e,i}$
above $\max I_0$ into $A$ as before.

We then continue by searching for the least number $s_1 > s_0$ so that
$\Phi_{e,s_1}(x)\converges$ for every $x \in I_0$ or some number less
than $u_0$ is enumerated into $C$ at stage $s_1$.  If no such number
$s_1$ exists, then let $s_1 = \infty$.  Set $g(e,i,s) = 0$ whenever $s_0
\leq s < s_1$.  If $s_1$ is infinite, then no other work is done on
$N_{e,i}$. (In this case, $N_e$ is met because $\Phi_e$ is not total.)
Suppose $s_1$ is finite (that is, this search succeeds).  There are
two cases.  First, suppose some number less than $u_0$ is enumerated
in $C$ at stage $s_1$.  We then have permission from $C$ to enumerate
numbers in $I_0$ into $A$.  Accordingly, we cancel the restraint on
$I_0$ and put $r_{\langle e,i \rangle, 2j'}$ into $A$ whenever $j \leq
j' \leq k$.  In this case the interval $I_0$ has become useless to us,
and we go back to our first step but now starting at stage $s_1$.  If
we find a stage $s_2 \geq s_1$ with $\Phi_{i,s_2}^{C_{s_2}}(i)
\converges$, say with use $u_1$, we choose a new interval $I_1$ of the
same form as before, but now with $\min(I_1) > s_2$ and proceed as
before with $I_1$ in place of $I_0$, and setting $g(e,i,s) = 0$ for
$s_1 \leq s < s_2$.

Now, suppose no number smaller than $u_0$ is enumerated into $C$ at $s_1$.  Then, $\Phi_{e,s_1}(x)\converges$ for all $x \in I_0$.  We are
now in a position to make progress on $N_e$ provided that $C$ later
permits us to change $A$ on $I_0$.  We then search for the least
number $s_2 \geq s_1$ so that some number less than $u_0$ is
enumerated in $C$ at stage $s_2$.  If there is no such number then let
$s_2 = \infty$.  We set $g(e,i,s) = 1$ whenever $s_1 \leq s < s_2$ in
order to force $C$ to give us the desired permission.  If $s_2$ is
infinite, then no other work is done on $N_{e,i}$. (In this case, we
have $\lim_s g(e,i,s) = 1 = C'(i)$.) Suppose $s_2$ is finite (that is,
this search succeeds).  We then declare the interval $I_0$ to be
\emph{successful} and cancel the restraint on $I_0$.  Since a number
smaller than $u_0 < \min(I_0)$ has now entered $C$, we have permission
to enumerate elements of $I_0$ into $A$.  So, for each $j \leq j' \leq
k$ put exactly one of $r_{\langle e, i \rangle, 2j'}$, $r_{\langle e,i
  \rangle, 2j' + 1}$ into $A$ so that $A$ and $\Phi_e$ differ on at
least one of these numbers.  (This ensures that at least half of the
elements of $I_0$ are in $A \triangle \Phi_e$ and hence that $\rho_m
(A \triangle \Phi_e) > \frac{\rho_m (R_{e,i})}{4}$ where $m = \max I_0
+1$.)  We now restart our process as above.  We continue in this
fashion, defining a sequence of intervals.  Note that, in general,
$g(e,i,s) = 1$ if at stage $s$ the most recently chosen interval has
been declared successful and we are awaiting a $C$-change below it,
and otherwise $g(e,i,s) = 0$.  

This strategy clearly succeeds if any of its searches fail, by the
parenthetical remarks in the construction.  Also, if there are
infinitely many successful intervals, it ensures that
$\overline{\rho}(A \triangle \Phi_e) \geq \frac{\rho(R_{e,i})}{4} >
0$, so $N_e$ is met.  If all searches are successful but there are
only finitely many successful intervals, then $C'(i) = 0 = \lim_s
g(e,i,s)$ and $N_{e,i}$ is met.  Only finitely many elements of $R_{e,i}$ are
permanently restrained from entering $A$ (namely the elements of the final
interval, if any), so $\rho(A) = \frac{1}{2}$ for reasons already given.
\end{proof}

We now obtain the following from Proposition \ref{c.e.density} and Theorem \ref{thm:nonlow}.

\begin{corollary} 
If $\mathbf{a}$ is a c.e. degree, then $\mathbf{a}$ is low if and only if every c.e. set in $\mathbf{a}$ that has a density is coarsely computable.  
\end{corollary}

For an application of this result to a degree structure arising from
the notion of coarse computability, see Hirschfeldt, Jockusch, Kuyper,
and Schupp \cite{HJKS}.


\section{Density, $1$-genericity, and  randomness}

  As we have already mentioned, it is easily seen that every degree
contains a set that is both coarsely computable and generically
computable, and every nonzero degree contains a set with neither of
these properties.  On the other hand, the next two results show that
for ``most'' degrees $\mathbf{a}$, every $\mathbf{a}$-computable
set that is generically computable is also coarsely computable. 
A set $A$ is called \emph{$1$-generic} if for every c.e.\ set $S$ of binary strings, 
$A$ either meets or avoids $S$.

\begin{theorem} Let $A$ be a $1$-generic set and let $r \in [0,1]$.
  Suppose that $B \leq\sub{T} A$ and $B$ is partially computable at density $r$.
  Then $B$ is coarsely computable at density $r$.
\end{theorem}

\begin{proof}
Fix a Turing functional $\Phi$ with $B = \Phi^A$ and a partial computable
function $\phi$ such that $\phi(n) = B(n)$ for all $n$ in the domain of $\phi$, and $\underline{\rho}(\dom\phi) \geq r$. Let
$$S = \{\sigma \in 2^{<\omega} : \Phi^\sigma \hbox{ is incompatible with } \phi\}.$$
Then $S$ is a c.e.\ set of strings so $A$ either meets or avoids $S$.
If $A$ meets $S$, then $B$ disagrees with $\phi$ on some argument, a
contradiction.  Hence $A$ avoids $S$.  Fix a string $\gamma \prec
A$ such that no string extending $\gamma$ is in $S$.  Now define a
computable set $C$ as follows.  Given $n$, search for a string
$\sigma$ extending $\gamma$ such that $\Phi^\sigma(n) \converges$ and
put $C(n) = \Phi^\sigma(n)$ for the first such $\sigma$ that is
found.  Then $C$ is total because $A$ extends $\gamma$ and $\Phi^A$ is
total. Hence $C$ is a computable set.  Further, if $\phi(n)
\converges$ then $B(n) = C(n)$ since no extension of $\gamma$ is in
$S$.  Hence $C \triangledown B \supseteq \dom\phi$, so
$\underline{\rho}(C \triangledown B) \geq r$, and hence $B$ is
coarsely computable at density $r$.
\end{proof}

\begin{corollary} If $A$ is $1$-generic and $B \leq\sub{T} A$ is
  generically computable, then $B$ is coarsely computable.
\end{corollary}

We do not need the definition of $n$-randomness here, but we simply point 
out the easy result that if $A$ is $1$-random, then $\gamma(A) = \frac{1}{2}$.
A set $A$ is called \emph{weakly $n$-random} if $A$ does not belong to
any $\Pi^0_n$ class of measure $0$.

\begin{theorem}
\begin{itemize}
\item[(i)] If $A$ is weakly $1$-random, $B \leq\sub{tt} A$, and $B$ is
  partially computable at density $r$, then $B$ is coarsely computable at
  density $r$.

\item[(ii)]   If $A$ is weakly $2$-random, $B \leq\sub{T} A$, and $B$ is
partially computable at density $r$, then $B$ is coarsely computable at density $r$.
 
\end{itemize}
\end{theorem}

\begin{proof}
  To prove (i), fix a Turing functional $\Phi$ such that $B = \Phi^A$
  and $\Phi^X$ is total for all $X \subseteq \omega$.  Let $\phi$
  be a partial computable function that witnesses that $B$ is
  partially computable at density $r$, and define
$$P = \{X : \Phi^X  \hbox{ is compatible with } \phi\}.$$
Then $P$ is a $\Pi^0_1$ class and $A \in P$, so $\mu(P) > 0$, where
$\mu$ is Lebesgue measure.  By the Lebesgue density theorem, there is
a string $\gamma$ such that $\frac{\mu(P \cap
  [\gamma])}{\mu([\gamma])} > .6$, where $[\gamma]=\{X \in 2^\omega :
\gamma \prec X\}$.  Define
$$C = \left\{n : \frac{\mu(\{Z \succ \gamma : \Phi^Z(n) = 1\})}{\mu([\gamma])} \geq .5\right\}.$$
Then it is easily seen that $C$ is a computable set and $C
\triangledown B$ contains the domain of $\phi$, so $B$ is coarsely
computable at density $r$.

To prove (ii), fix a Turing functional $\Phi$ with $B = \Phi^A$ and
fix a partial computable function $\phi$ that witnesses that $B$
is partially computable at density $r$.  Define
$$P = \{ X : \Phi^X \hbox{ is total and compatible with }
\phi \}.$$ Then $P$ is a $\Pi^0_2$ class and $A \in P$, so $\mu(P) >
0$.  Then for notational convenience assume that $\mu(P) > .8$,
applying the Lebesgue density theorem as in part (a).  It follows that
for every $n$ there exists $i \leq 1$ such that $\mu(\{X : \Phi^{X}(n)
= i\}) \geq .4$.  Given $n$, one can compute such an $i$ effectively,
and then put $n$ into $C$ if and only if $i = 1$.  One can easily
check that $C$ is computable and $C \triangledown B \supseteq
\dom\phi$, so $\underline{\rho}(C \triangledown B) \geq
\underline{\rho}(\dom\phi) \geq r$.  Hence $B$ is coarsely computable
at density $r$.
\end{proof}

Note that $1$-randomness does not suffice in part (ii) of the above
theorem, since every set is computable from some $1$-random set.

Since the $1$-generic sets are comeager and the weakly $2$-generic
sets have measure $1$, it follows from the last two theorems that
generic computability implies coarse computability below almost every
set, both in the sense of Baire category and in the sense of measure.
The next result contrasts with this fact.

\begin{theorem}
  If the degree $\mathbf{a}$ is hyperimmune, there is a set $B \leq\sub{T}
  A$ such that $B$ is bi-immune and of density $0$.
\end{theorem}

We omit the proof, which is an easy variation of Jockusch's proof in
\cite{J1}, Theorem 3, that every hyperimmune set computes a bi-immune set.

Bienvenu, Day, and H\"olzl \cite{BDH} proved the beautiful theorem that
every nonzero Turing degree contains an absolutely undecidable set
$A$; that is, a set such that every partial computable function that
agrees with $A$ on its domain has a domain of density $0$.  We now
consider the degrees of sets that are both absolutely undecidable and
coarsely computable.

\begin{corollary}
  In the sense of Lebesgue measure, almost every set $A$ computes a
  set $B$ that is absolutely undecidable and coarsely computable.
\end{corollary}

\begin{proof} D. A. Martin (see \cite[Theorem 8.21.1]{DH}) proved
  that almost every set has hyperimmune degree.  It is obvious that
  every bi-immune set is absolutely undecidable.
\end{proof}

On the other hand, Gregory Igusa has proved the following theorem
using forcing with computable perfect trees.

\begin{theorem}[Igusa, to appear] There is a noncomputable set $A$
  such that no set $B \leq\sub{T} A$ is both coarsely computable and
  absolutely undecidable.
\end{theorem}

We now turn to studying the degrees of sets $A$ such that $\gamma(A) =
1$ but $A$ is not coarsely computable.  As shown in Theorem
\ref{degreeobs}, if either $\mathbf{a \nleq 0'}$ or $\mathbf{a}$ is a
nonzero c.e.\ degree, then $\mathbf{a}$ contains such a set.  This
observation might  lead one to conjecture that every nonzero degree
computes such a set, but  we shall prove the opposite for  $\Delta^0_2$
$1$-generic degrees.  We will reach this result by first considering sets for which $\gamma(A) = 1$ is witnessed constructively.  

\begin{definition} We say that $\gamma(A) = 1$ \emph{constructively}
  if there is a uniformly computable sequence of computable sets $C_0, C_1,
  \dots$ such that $\overline{\rho}(A \triangle C_n) < 2^{-n}$ for all $n$.
\end{definition}

Of course, if $A$ is coarsely computable, then $\gamma(A) = 1$
constructively.  Although the converse  appears  unlikely, it
was proved by Joe Miller.

\begin{theorem}[Joe Miller, private communication]
\label{constr} 
If
  $\gamma(A) = 1$ constructively, then $A$ is coarsely computable.
\end{theorem}

\begin{proof}
  We present  Miller's proof in  essentially  the form in which he gave it.
    Let $I_k$ be the interval $[2^k - 1, 2^{k+1} - 1)$.  For
  any set $C$, let $d_k(C)$ be the density of $C$ on $I_k$, so $d_k(C)
  = \frac{|C \cap I_k|}{2^k}$.  The following lemma, which will also
  be useful in the proof of Theorem \ref{1G}, relates $\overline{\rho}(C)$
to $\overline{d}(C)$, where $\overline{d}(C) = \limsup_k d_k(C)$.

\begin{lemma}\label{factor2}   For every set $C$, 
$$\frac{\overline{d}(C)}{2} \leq
\overline{\rho}(C) \leq 2 \overline{d}(C).$$
\end{lemma}

\begin{proof} 
  For all $k$,
$$d_k(C) = \frac{|C \cap I_k|}{2^k} \leq \frac{|C \upharpoonright(2^{k+1} -1)|}
{2^k} \leq 2 \rho_{2^{k+1} - 1} (C).$$ 
Dividing both sides of this inequality by $2$ and then taking the lim
sup of both sides yields that $\frac{\overline{d}(C)}{2} \leq \overline{\rho}(C)$.

To prove that $\overline{\rho}(C) \leq 2 \overline{d}(C)$, assume that
$k - 1 \in I_n$, so $2^n \leq k < 2^{n+1}$.  Then
$$\rho_k(C) = \frac{|C \upharpoonright k|}{k} \leq 
\frac{|C \upharpoonright (2^{n+1} - 1)|}{2^n} = \frac{\sum_{0 \leq i
    \leq n} 2^i d_i(C)}{2^n} < 2 \max_{i \leq n} d_i(C).$$ 

Let $\epsilon > 0$ be given.  Then $d_i(C) < \overline{d}(C) +
\epsilon$ for all sufficiently large $i$.  Hence there is a finite set
$F$ such that $d_i(C \setminus F) < \overline{d}(C \setminus F) + \epsilon$
for \emph{all} $i$.  Then, by the above inequality applied to $C\setminus F$,
we have $\rho_k(C\setminus F) < 2 (\overline{d}(C\setminus F) + \epsilon)$ for all $k$,
so $\overline{\rho}(C\setminus F) \leq 2\overline{d}(C\setminus F)$.  As
$\overline{\rho}$ and $\overline{d}$ are invariant under finite
changes of their arguments and $\epsilon > 0$ is arbitrary, it follows
that $\overline{\rho}(C) \leq 2 \overline{d}(C)$.
\end{proof}
 
We now complete the proof of Theorem \ref{constr}.  Let the sequence
$C_n$ witness that $\gamma(A) = 1$ constructively, so $\{C_n\}$ is
uniformly computable and $\overline{\rho}(A \triangle C_n) < 2^{-n}$
for all $n$.  It follows from the lemma that $\overline{d}(A \triangle
C_n) < 2^{-n +1}$.  Hence, for each $n$, if $k$ is sufficiently large,
we have $d_k (A \triangle C_n) < 2^{-n + 1}$.

For $m < n$, we say that $C_m$ \emph{trusts} $C_n$ on $I_k$ if
$d_k(C_n \triangle C_m) < 2^{-m + 2}$.  We say that $C_n$ is
\emph{trusted} on $I_k$ if $C_m$ trusts $C_n$ for all $m < n$.  Note
that $C_0$ is trusted on every interval $I_k$.  We now define a
computable set $C$ that will witness that $A$ is coarsely computable.
For each $k$, let $N \leq k$ be maximal such that $C_N$ is trusted on
$I_k$, and let $C \upharpoonright I_k = C_N \upharpoonright I_k$.

We claim that $\rho(A \triangle C) = 0$.  Fix $n$.  Let $k \geq n$ be
large enough that $d_k(A \triangle C_m) < 2^{-m + 1}$ for all $m \leq
n$.  Then $d_k (C_n \triangle C_m) \leq d_k(A \triangle C_n) + d_k(A
\triangle C_m) < 2^{-m +1} + 2^{-n +1} < 2^{-m + 2}$ for all $m < n$.
Therefore, $C_n$ is trusted on $I_k$.  Hence $C \upharpoonright I_k =
C_N \upharpoonright I_k$ for some $N \geq n$ such that $C_N$ is
trusted on $I_k$.  Therefore, $C_n$ trusts $C_N$ on $I_k$, so $d_k
(C_n \triangle C_N) < 2^{-n + 2}$.  It follows that $d_k(A \triangle
C) = d_k(A \triangle C_N) \leq d_k(A \triangle C_n) + d_k(C_n
\triangle C_N) < 2^{-n+1} + 2^{-n+2} < 2^{-n + 3}$.  Because this is
true for every sufficiently large $k$, we have $\overline{d}(A
\triangle C) \leq 2^{-n+3}$.  Since $n$ was arbitrary, it follows that
$\overline{d}(A \triangle C) = 0$ and hence, by the lemma, $\rho(A
\triangle C) = 0$. Thus $A$ is coarsely computable.
\end{proof}

\begin{corollary} \label{0'approx}
  Suppose there is a $0'$-computable function $f$ such that, for all
  $e$, we have that $\Phi_{f(e)}$ is total and $\{0,1\}$-valued, and
  $\overline{\rho}(A \triangle \Phi_{f(e)}) \leq 2^{-e}$.
    Then $A$ is coarsely computable.
\end{corollary}

\begin{proof}
  By the theorem, it suffices to show that $\gamma(A) = 1$
  constructively.  Let $g$ be a computable function such that $f(e) =
  \lim_s g(e,s)$.  We now define a computable function $h$ such that,
  for all $e$, we have that $\Phi_{h(e)}$ is total and differs on only finitely
  many arguments from $\Phi_{f(e)}$, so that $\Phi_{h(0)},
  \Phi_{h(1)}, \dots$ witnesses that $\gamma(A) = 1$ constructively.
  To compute $\Phi_{h(e)}(n)$, search for $s \geq n$ such that
  $\Phi_{g(e,s)}(n)$ converges in at most $s$ many steps, and let
  $\Phi_{h(e)}(n) = \Phi_{g(e,s)}(n)$.  The $s$-$m$-$n$ theorem gives us
  such an $h$, and clearly $h$ has the desired properties.
\end{proof}

We now have the tools to prove the following result, which we
did not initially expect to be true.

\begin{theorem} \label{1G} Let $G$ be a $\Delta^0_2$ $1$-generic set,
  and suppose that $A \leq\sub{T} G$ and $\gamma(A) = 1$. Then $A$ is
  coarsely computable.
\end{theorem}

\begin{proof}
Fix $\Phi$ such that $A = \Phi^G$. As in the proof of Theorem
\ref{constr} let $I_k$ be the interval $[2^k - 1, 2^{k+1} - 1)$ and
define $d_k(C) = \frac{|C \upharpoonright I_k|}{2^k}$ and
$\overline{d}(C) = \limsup_k d_k(C)$.

Consider first the case that for some $\epsilon > 0$ and for every
computable set $C$ and every number $k$, we have that $G$ meets the
set $S_{\epsilon, C, k}$ of strings defined below:
$$S_{\epsilon, C, k}  = \{\nu : (\exists l > k)
[d_l (\Phi^\nu \triangle C) \geq \epsilon]\}.$$

Of course, $\nu$ must be such that $\Phi^\nu(j) \converges$ for all $j
\in I_l$ for the above to make sense. We claim that $\gamma(A) < 1$
in this case, so that this case cannot arise.  Let $C$ be a computable
set and fix $\epsilon$ as in the case hypothesis. Then, for every $k$
there exists $l > k$ such that $d_l (A \triangle C) \geq \epsilon$ by
the choice of $\epsilon$. It follows that $\overline{d}(A \triangle
C) \geq \epsilon$, so $\overline{\rho}(A \triangle C) \geq
\frac{\epsilon}{2}$ by Lemma \ref{factor2}.  By Lemma \ref{comp} it
follows that $\underline{\rho}(A \triangledown C) \leq 1 -
\frac{\epsilon}{2}$. Hence $\gamma(A) \leq 1 - \frac{\epsilon}{2} <
1$. Since $\gamma(A) = 1$ by assumption, this case cannot arise.

Since $G$ is $1$-generic, it follows that for every $n$ there is a
computable set $C$ and a number $k$ such that $G$ avoids
$S_{2^{-(n+2)}, C, k}$; i.e., there exists $\gamma \prec G$ such
that $\gamma$ has no extension in $S_{2^{-(n+2)}, C, k}$. Given $l
\geq k$, let $\nu_0$ and $\nu_1$ be strings extending $\gamma$ such
that $\Phi^{\nu_i}(x)\converges$ for all $x \in I_l$ and $i \leq
1$. Then
$$d_l (\Phi^{\nu_0} \triangle \Phi^{\nu_1}) \leq
d_l (\Phi^{\nu_0} \triangle C) + d_l(C \triangle \Phi^{\nu_1}) <
2^{-(n+2)} + 2^{-(n+2)} = 2^{-(n+1)}.$$ Since $G$ is $\Delta^0_2$,
using an oracle for $0'$ we can find $\gamma_n$ and $k_n$ such that
for all $\nu_0,\nu_1$ extending $\gamma_n$ and all $l \geq k_n$, if
$\Phi^{\nu_i}(x)\converges$ for all $x \in I_l$ and $i \leq 1$ then
$d_l (\Phi^{\nu_0} \triangle \Phi^{\nu_1}) \leq 2^{-(n+1)}$. Note that
if we take $\nu_0 \prec G$ then $d_l(\Phi^{\nu_1} \triangle A) <
2^{-(n+1)}$.

For each $n$, define a computable set $B_n$ as follows.  On each
interval $I_k$ search for $\nu_1 \succcurlyeq \gamma_n$ such that
$\Phi^{\nu_1}$ converges on $I_k$.  Note that such a $\nu_1$ exists
because $\gamma_n \prec G$ and $\Phi^G$ is total. Let $B_n
\upharpoonright I_k = \Phi^{\nu_1} \upharpoonright I_k$. Then $B_n$ is
a computable set, since the only non-effective part of its definition
is the use of the \emph{single} string $\gamma_n$.  Furthermore, an
index for $B_n$ as a computable set can be effectively computed from $\gamma_n$ and hence from $0'$.

We claim that $\overline{\rho}(B_n \triangle A) \leq 2^{-n}$.  Fix
$n$. By Lemma \ref{factor2}, it suffices to show that
$\overline{d}(B_n \triangle A) \leq 2^{-(n+1)}$.  For all $k$, we have
that $d_k(B_n \triangle A) = d_k(\Phi^{\nu_1} \triangle A)$ for some
string $\nu_1$ extending $\gamma_n$. Hence, if $k$ is sufficiently
large, it follows that $d_k(B_n \triangle A) \leq 2^{-n+1}$, and hence
$\overline{d}(B_n \triangle A) \leq 2^{-(n+1)}$, so
$\overline{\rho}(B_n \triangle A) \leq 2^{-n}$. It now follows from
Corollary \ref{0'approx} with $\Phi_{f(e)} = B_e$ that $A$ is coarsely
computable.
\end{proof}

\section{Further results}

In this section we investigate the complexity of $\gamma(A)$ as a real
number when $A$ is c.e.\ and look at the distribution of values of
$\gamma(B)$ as $B$ ranges over all sets computable from a given set
$A$. A real is \emph{left-$\Sigma^0_3$} if $\{q\in \mathbb Q:q<r\}$ is
$\Sigma^0_3$.

\begin{proposition} \label{sigma3} If $A$ is a c.e.\ set, then
  $\gamma(A)$ is a left-$\Sigma^0_3$ real.
\end{proposition}

\begin{proof}
  Let $A$ be a c.e.\ set, and let $q$ be a rational number with $q
  \neq \gamma(A)$.   Then the following two statements are equivalent:

\begin{enumerate}
     \item[(i)] $q < \gamma(A)$.

     \item[(ii)]   There is a computable set $C$ such that $\rho_n(A
       \triangledown C) \geq q$  for \emph{all} $n$.
\end{enumerate}

It is immediate that (ii) implies (i) since (ii)
implies that $A$ is coarsely computable at density $q$ and hence $q
\leq \gamma(A)$.

Now assume (i) in order to prove (ii).  Let $r$ be a real number with
$q < r < \gamma(A)$.  Then $A$ is coarsely computable at density $r$,
so there is a computable set $C$ such that $A \triangledown C$ has
lower density at least $r$.  Since $q < r$, it follows that $\rho_n(A
\triangledown C) \geq r$ for all but finitely many $n$.  By making a
finite change in $C$, we can ensure that this inequality holds for
\emph{all} $n$.

Routine expansion shows that the set of rational numbers $q$ satisfying (ii) is
$\Sigma^0_3$, so $A$ is left-$\Sigma^0_3$ by definition.

\emph{Note:} The formulation of (ii) was chosen in order to minimize
the number of quantifiers when it is expanded.  If we proceeded
by simply using the fact that, for $q \neq \gamma(A)$, we have that $q
< \gamma(A)$ if and
only if $A$ is coarsely computable at density $q$ and used a routine
expansion of the latter, we could conclude only that $\gamma(A)$ is
left-$\Sigma^0_5$.
\end{proof}

In the next result, we prove the converse and thus characterize the
reals of the form $\gamma(A)$ for $A$ c.e.

\begin{theorem}
 Suppose $0 \leq r \leq 1$.  Then the
  following are equivalent:
\begin{itemize}
\item[(i)] $r = \gamma(A)$ for some c.e.\ set $A$.
     \item[(ii)] $r$ is left-$\Sigma^0_3$.
\end{itemize}
\end{theorem}

\begin{proof}
It was shown in the previous proposition that (i) implies (ii), so it
remains to be shown that (ii) implies (i). Let $r$ be
left-$\Sigma^0_3$. Our proof is based on that of Theorem 5.7 of
\cite{DJS}, which shows that $r$ is the lower density of some
c.e.\ set. That proof consists in taking a $\Delta^0_2$ set $B$ such
that $\underline{\rho}(B)=r$ (which exists by the relativized form of
Theorem 5.1 of \cite{DJS}) and constructing a strictly increasing
$\Delta^0_2$ function $t$ and a c.e.\ set $A$ such that for each $n$,
\begin{enumerate}

\item $\rho_{t(n)}(A)=\rho_n(B)$ 

\item $A \cap [t(n),t(n+1))$ is an initial segment of $[t(n),t(n+1))$.

\end{enumerate}
It then follows easily that $\underline{\rho}(A)=\underline{\rho}(B)=r$.

Let $S$ be the set of all pairs $(k,e)$ such that $e \leq k$. Let $f$
be a computable bijection between $S$ and $\omega$. We can easily adapt the proof
of Theorem 5.7 of \cite{DJS} to replace (1) by
\begin{itemize}

\item[(1$^\prime$)] $\rho_{t(f(k,e))}(A)=\rho_k(B)$ for each $k$ and $e
\leq k$,

\end{itemize}
while still having (2) hold for each $n$. Furthermore, we can also
ensure that when a new approximation $t(n,s+1)$ to $t(n)$ is defined,
it is chosen to be greater than both $2^{t(n-1,s+1)}$ and $2^{t(s,s)}$
(because for each instance of Lemma 5.8 of \cite{DJS}, there are
infinitely many $c$ witnessing the truth of the lemma).

We now define a c.e.\ set $C$. At stage $s$, proceed as
follows for each pair $(k,e)$ with $f(k,e) \leq s$. Let $n=f(k,e)$. If
$\Phi_{e,s}(x)\converges$ for all $x \in [t(n-1,s),t(n,s))$, then for
each such $x$ for which $\Phi_e(x)=0$, enumerate $x$ into $C$ (if
$x$ is not already in $C$). We say that $x$ is put into $C$ for the
sake of $(k,e)$.

Let $D=A \cup C$. Then $D$ is a c.e.\ set, and $\underline{\rho}(D)
\geq \underline{\rho}(A) = r$. By Theorem 3.9 of \cite{DJS}, for each
$\epsilon>0$, there is a computable subset of $D$ with lower density
greater than $r-\epsilon$. It follows that $\gamma(D) \geq r$.

Now let $e$ be such that $\Phi_e$ is total. Fix a $k$ and let
$n=f(k,e)$. Let $s$ be least such that $t(n,s+1)=t(n)$. Every number
put into $C$ by the end of stage $s$ is less than $t(s,s)$. Every
number put into $C$ after stage $s$ for the sake of any pair other
than $(k,e)$ is either less than $t(n-1)=t(n-1,s+1)$ or greater than
or equal to $t(n)$. By our assumption on the size of $t(n)$, it
follows that $C(x) \neq \Phi_e(x)$ for every $x \in [\log_2
t(n),t(n))$, so $\rho_{t(n)}(C \triangledown \Phi_e) \leq
\frac{\log_2 t(n)}{t(n)}$, and hence
\begin{multline*}\rho_{t(n)}(D \triangledown \Phi_e) \leq
\rho_{t(n)}(C \triangledown
\Phi_e) + \rho_{t(n)}(D \triangledown C) \\ \leq \frac{\log_2
t(n)}{t(n)} + \rho_{t(n)}(A)=\frac{\log_2
t(n)}{t(n)} + \rho_k(B).
\end{multline*}
Since $\lim_n \frac{\log_2 t(n)}{t(n)} =0$, we have
$\underline{\rho}(D \triangledown \Phi_e) \leq
\underline{\rho}(B)=r$. Since $e$ is arbitrary, $\gamma(D) \leq r$.
\end{proof}


\begin{definition}
  If $A \subseteq \N$ we call  
\[S(A) = \{\gamma(B) : B \leq\sub{T} A\}  \subseteq [0,1]\]
the \emph{coarse spectrum} of $A$.
\end{definition}

\begin{theorem} For any set $A$ and any $\Delta^0_2$ real $s \in
  [0,1]$, we have that $s \cdot \gamma(A) + (1-s) \in S(A)$.  It follows that $S(A)$ is
  dense in the interval $[\gamma(A), 1]$.
\end{theorem}

\begin{proof} 
We may assume that $s>0$, since any computable $B \leq\sub{T} A$
witnesses the fact that $1 \in S(A)$.
By Theorem 2.21 of \cite{JS} there is a computable set
  $R$ of density $s$. Note that $R$ is infinite.
  Let $h$ be an increasing computable function with range $R$, and let
  $B = h(A)$.  Then $B \leq\sub{T} A$, so it suffices to prove that
  $\gamma(B) = s \cdot \gamma(A) + (1-s)$.  For this, we need the
  following lemma, which relates the lower density of $h(X)$ to that
  of $X$.  The corresponding lemma for density was proved as Lemma 3.4
  of \cite{DJMS}, and the proof here is almost the same.

\begin{lemma} \label{prod} Let $h$ be a strictly increasing function
  and let $X \subseteq \omega$. Then $\underline{\rho}(h(X)) =
  \rho(\mbox{range}(h)) \underline{\rho}(X)$, provided that the range
  of $h$ has a density.
\end{lemma}

\begin{proof}
  Let $Y$ be the range of $h$, and for each $u$, let $g(u)$ be the
  least $k$ such that $h(k) \geq u$.  As shown in the proof of Lemma
  3.4 of \cite{DJMS}, $\rho_u (h(X)) = \rho_u (Y) \rho_{g(u)}(X) $ for
  all $u$, via bijections induced by $h$.  Taking the lim inf of both
  sides and using the fact that $\rho(Y)$ exists, we see that
$$\underline{\rho}(h(X))
= \rho(Y)(\liminf \langle \rho_{g(0)}(X), \rho_{g(1)}(X) , \dots
\rangle).$$ It is easily seen that the function $g$ is finite-one and
$g(h(x)) = x$ for all $x$, and $g(u+1) \leq g(u) + 1$ for all $u$.
Hence the sequence on the right-hand side of the above equation can be
obtained from the sequence $\rho_0 (X), \rho_1(X), \dots$ by replacing
each term by a finite, nonempty sequence of terms with the same value.
Thus the two sequences have the same lim inf, and we obtain
$\underline{\rho}(h(X)) = \rho(Y) \underline{\rho}(X)$, as needed
to prove the lemma.
\end{proof}

To prove that $\gamma(B) = s \cdot \gamma(A) + (1-s)$, it suffices to
show that for each $t \in [0,1]$, $A$ is coarsely computable at
density $t$ if and only if $B$ is coarsely computable at density $st +
1 - s$.  Suppose first that $A$ is coarsely computable at density $t$,
and let $C$ be a computable set such that $\underline{\rho}(A
\triangledown C) \geq t$.  Let $\widehat{C} = h(C) \cup \overline{R}$.   Then
$\widehat{C}$ is a computable set and
$$\underline{\rho}(\widehat{C} \triangledown B) = \underline{\rho}(h(C \triangledown A) \cup \overline{R}) \geq \underline{\rho}(h(C \triangledown A)) +
\underline{\rho}(\overline{R}) = s \underline{\rho}(C \triangledown A)
+ 1 - s \geq s \cdot t + (1-s).$$
It follows that $B$ is coarsely computable at density $st + (1-s)$.   

Conversely, if a computable set $\widehat{C}$ witnesses that $B$ is
coarsely computable at density $st + (1-s)$, let $C =
h^{-1}(\widehat{C})$, and check by a similar argument that $C$ witnesses
that $A$ is coarsely computable at density $t$ since $s > 0$.
\end{proof}

\end{document}